\documentclass[11pt]{amsart}

\usepackage{amsfonts, amssymb, amscd}
\usepackage{verbatim}
\usepackage{amssymb}
\usepackage{mathrsfs}
\usepackage{graphicx}
\usepackage[nospace,noadjust]{cite}
\usepackage[all]{xy}
\usepackage{tikz}
\usepackage{subfiles}
\usepackage[toc,page]{appendix}

\usepackage{hyperref}

\newcommand{\Zz}{\mathbb{Z}}

\newcommand{\Pp}{\mathbb{P}}
\newcommand{\Rr}{\mathbb{R}}
\newcommand{\Qq}{\mathbb{Q}}
\newcommand{\Nn}{\mathbb{N}}

\newcommand{\Kk}{\mathbb{K}}
\newcommand{\NN}{\mathcal{N}}
\newcommand{\RR}{\mathfrak{R}}
\newcommand{\Ss}{\mathfrak{S}}
\newcommand{\Oo}{\mathcal{O}}
\newcommand{\BB}{\mathbb{B}}

\newcommand{\Bb}{\mathcal{B}}

\newcommand{\lcm}{\operatorname{lcm}}
\newcommand{\Supp}{\operatorname{Supp}}
\newcommand{\LCS}{\operatorname{LCS}}

\newcommand{\dive}{\operatorname{div}}
\newcommand{\modu}{\operatorname{mod}}
\newcommand{\la}{\langle}
\newcommand{\ra}{\rangle}
\newcommand{\lf}{\lfloor}
\newcommand{\rf}{\rfloor}

\newtheorem{theorem}{Theorem}[section]
\newtheorem{proposition}[theorem]{Proposition}
\newtheorem{definition}[theorem]{Definition}
\newtheorem{lemma}[theorem]{Lemma}
\newtheorem{corollary}[theorem]{Corollary}
\newtheorem{remark}{Remark}
\newtheorem{notation}{Notation}
\newtheorem{construction}{Construction}

\begin{document}

\title{Boundedness of $n$-complements on surfaces}
\begin{abstract}
In light of Shokurov's result \cite[Theorem 3]{Sh20} in 2020, the main purpose of this paper is to prove the boundedness of $n$-complements for surface pairs $(X/Z\ni o,B)$ under the same setting, without the asssumptions that $X/Z$ is of weak Fano type or $(X/Z\ni o,B)$ has a klt $\Rr$-complement.
\end{abstract}
\author{Xiangze Zeng}
\maketitle
\tableofcontents

\section{Introduction}

The concept of $n$-complements naturally arises from the study on the anticanonical linear system $|-nK_X|$, and was introduced by V.V.Shokurov in \cite{Sh93}. Later this concept of $n$-complements was generalized to $\Rr$-complements and Shokurov asked: does the existence of an $\Rr$-complement necessarily imply the existence of an $n$-complement for some integer $n$? And if so,what kind of restrictions can we put on this $n$? The widely expected result is that there exists a finite set $\NN$ of positive integers such that, for a certain class of relative pairs $(X/Z\ni o,B)$, as long as $(X/Z\ni o,B)$ has an $\Rr$-complement, it also has an $n$-complement for some $n\in\NN$. We'll refer to this statement as "boundedness of $n$-complements" throughout this paper. Boundedness of $n$-complements on surfaces was established by V.V.Shokurov in \cite{Sh97}. In \cite{Birkar19} C.Birkar proved boundedness of $n$-complements for pairs $(X,B)$ of Fano type with hyperstand multiplicities in arbitrary dimensions, and later Shokurov generalized this boundedness result to all relative pairs of weak Fano type in \cite{Sh20}. In this paper, we'll generalize Shokurov's boundedness result Theorem 3 in \cite{Sh20} to all surface pairs $(X/Z\ni o,B)$, without the weak Fano type assumption on $X/Z$ and without the restrictions on the multiplicities of $B$. To better state our main theorem \ref{main}, we recall the formal definitions of complements first:

\begin{definition}\label{R-comp}
\textup{Let $(X/Z\ni o,D)$ be a pair with a proper local morphism $X/Z\ni o$ and an $\Rr$-divisor $D$. Then a pair $(X/Z\ni o,D^+)$ with the same local morphism is called an} $\Rr$-complement \textup{of $(X/Z\ni o,D)$ if the following holds:
\begin{enumerate}
\item $D^+\geq D$;
\item $(X,D^+)$ is lc;
\item $K_X+D^+\sim_{\Rr}0/Z\ni o$.
\end{enumerate}
}
\end{definition}

\begin{definition}\label{n-comp}
\textup{Let $(X/Z\ni o,D)$ be a pair with a proper local morphism $X/Z\ni o$ and an $\Rr$-divisor $D=\sum d_iD_i$, and $n$ is a positive integer. Then a pair $(X/Z\ni o,D^+=\sum d_i^+D_i)$ with the same local morphism is called an} $n$-complement \textup{of $(X/Z\ni o,D)$ if the following holds:
\begin{enumerate}
\item for every prime divisor $D_i$,
$$ d_i^+\geq
\begin{cases}
1& \text{if $d_i=1$;}\\
\frac{\lf (n+1)d_i\rf}{n}&\text{otherwise;}
\end{cases}$$
\item $(X,D^+)$ is lc;
\item $K_X+D^+\sim_{n}0/Z\ni o$.
\end{enumerate}
}
\end{definition}

Now we are ready to state our main result Theorem \ref{main}, with the complementary index $n$ under the restrictions given by Definition \ref{comp-restriction} below.

\begin{definition}\label{comp-restriction}
\textup{Given these data: a positive integer $I$, a nonrational vector $v\in\Rr^l$, a direction $e$ of the rational affine span $\la v\ra$ and a real number $\epsilon>0$. We say an integer $n$ is under} complementary restrictions \textup{with respect to $I,v,e,\epsilon$ if there exists a vector $v_n\in\Rr^l$ satisfying the following:
\begin{enumerate}
\item Divisibility: $I|n$;
\item Denominator with $n$: $nv_n\in\Zz^l$;
\item Approximation: $||v_n-v||<\epsilon/n$;
\item Directional Approximation: $||\frac{v_n-v}{||v_n-v||}-e||<\epsilon$.
\end{enumerate}
}
\end{definition}

\begin{theorem}\label{main}
Suppose we have the data $I,v,e,\epsilon$ as in Definition \ref{comp-restriction}. Then we can find a finite set $\NN=\NN(I,v,e,\epsilon)$ such that:
\begin{enumerate}
\item {\rm Restrictions:} $\NN$ consists of integers under complementary restrictions;
\item {\rm Existence of $n$-complements:} suppose we have a surface pair $(X/Z\ni o,B)$ with a proper local morphism $X/Z\ni o$, connected central fiber $X_o$ and an $\Rr$-boundary $B$. Then if $(X/Z\ni o,B)$ has an $\Rr$-complement, it also has an $n$-complement for some $n\in\NN$.
\end{enumerate}
\end{theorem} 

I would like to thank Professor Shokurov for his help and guidance. This work wouldn't be completed without his deep insights and infinite patience. 

\section{Preliminaries}

\subsection{Basic Properties of Complements}
In this section we list some basic properties of $n$-complements frequently used throughout this paper. Lemma \ref{pushforwards} implies that $n$-complements are preserved by birational pushforwards, while Lemma \ref{monotonicity} is a monotonicity result for $n$-complements.

\begin{lemma}\label{pushforwards}
(\cite[Proposition 4.3.1]{Pr01}). Let $f:Y/Z\to X/Z$ be a birational contraction, $D$ a subboundary on $Y$. If $(Y/Z,D)$ has an $n$-complement $(Y/Z,D^+)$, then $(X/Z,f(D))$ has an $n$-complement $(X/Z,f(D^+))$.
\end{lemma}

\begin{remark}\label{pullback}
\textup{If we put certain restrictions on the coefficients of the boundary $D$, we can also pull back complements along a birational contraction. See \cite[Proposition 4.3.2]{Pr01} for one of such result.}
\end{remark}

%\begin{lem}\label{pullback}
%(\cite[Proposition 4.3.2]{https://doi.org/10.48550/arxiv.math/9912111}). Fix $n\in\mathbb{N}^{+}$ and define $\mathcal{P}_n:=\{\alpha|0\leq\alpha\leq1, \lfloor(n+1)\alpha\rfloor\geq n\alpha\}$. Let $f:Y/Z\rightarrow X/Z$ be a birational contraction, $D$ a subboundary on $Y$ such that
%\begin{enumerate}[(1)]
%\item $K_Y+D$ is nef over $X$;
%\item $f(D)$ has coefficients in $\mathcal{P}_n$.
%\end{enumerate}
%Then if $(X/Z,f(D))$ has an $n$-complement, $(Y/Z,D)$ also has an $n$-complement.
%\end{lem}

%\begin{lem}\label{pullback-0-pair}
%Let $f:(Y,B)\rightarrow (X,f(B))$ be a birational contraction of surface pairs, $B$ is a boundary and $K_Y+B=f^*(K_X+f(B))$. Then an $n$-complement $(X,f(B)^+)$ of $(X,f(B))$ with $f(B)^+\geq f(B)$ can be pulled back to an $n$-complement $(Y,B^+)$ of $(Y,B)$ with $B^+\geq B$.
%\end{lem}
%\begin{proof}
%Let $K_Y+B^+=f^*(K_X+f(B)^+)$. Then $B^+-B=f^*(f(B)^+-f(B))$ and $f(B)^+-f(B)\geq0$, by Negativity Lemma [KM98, Corollary 3.39] $B^+\geq B$. It remains to check if $n(K_Y+B^+)\sim0$ and $B^+\geq B$, then $K_Y+B^+$ is an $n$-complement of $K_Y+B$. Now let $P$ be a prime 
%divisor in $B$, and $b^+$, $b$ be its coefficient in $B^+$ and $B$ respectively. If $b^+=1$, either $b=1$ or $b<1$ in which case $b^+=1\geq\lfloor (n+1)b \rfloor/n$. Otherwise assume $b^+=r/n$ with $r<n$. Then $b^+=\lfloor (n+1)b^+\rfloor/n\geq \lfloor (n+1)b\rfloor/n$.
%\end{proof}

\begin{lemma}\label{monotonicity}
(\cite[Proposition 1]{Sh20}). Let $(X/Z,D)$ and $(X/Z,D\,')$ be two pairs on $X$ with $D\geq D\,'$. If $(X/Z,D)$ has an $n$-complement $(X/Z,D^+)$, then $(X/Z,D^+)$ is also an $n$-complement of $(X/Z,D\,')$.
\end{lemma}

\subsection{Rational Approximation of Coefficients}

We'll quote and use certain known boundedness results on $n$-complements for the main proof. Since these positive results are established under the assumption that coefficients of $B$ are contained in a hyperstandard set, or more generally, a dcc set of rational numbers in [0,1], we briefly introduce the definition of hyperstandard sets and a procedure of lower rational approximation of coefficients using hyperstandard sets here. For more details of related constructions, properties and proofs, see Construction 1 from \cite{Sh20}.  

\begin{definition}\label{hyperstandard-sets}
(\cite[3.2]{PSh09}). \textup{Let $\RR$ be a set of rational numbers in $[0,1]$. Then a hyperstandard set is $$\Phi=\Phi(\RR)=\{1-\frac{r}{l}|r\in\RR, \text{$l$ is a positive integer}\}\cap[0,1].$$}
\end{definition}

\begin{remark}\label{standard-multiplicities}
\textup{This definition of hyperstandard sets is a generalization of the set $S$ of standard multiplicities where $S=\{1\}\cup\{1-1/l|l\in\Nn^+\}=\Phi(\{0,1\})$. Hence to make sure that the hyperstandard set $\Phi(\RR)$ contains $S$, we'll always assume $\{0,1\}\subset\RR$.}
\end{remark}

For a finite set $\NN$ of positive integers, and a hyperstandard set $\Phi=\Phi(\RR)$ as in Definition \ref{hyperstandard-sets}, we can construct a new set $\Gamma(\NN,\Phi)$ of rational numbers in [0,1] as follows: $$\Gamma(\NN,\Phi)=\{1-\frac{r}{l}+\frac{1}{l}(\sum\limits_{n\in\NN}\frac{m_n}{n+1})|r\in\RR, l\in\Nn^+, m_n\in\Nn\}\cap[0,1].$$ $\Gamma(\NN,\Phi)$ is again a hyperstandard set by the following lemma.

\begin{lemma}\label{prop-hyperstandard}
(\cite[Corollary 4]{Sh20}). Let $\NN$ be a nonempty finite set of positive integers and $\RR$ be a finite set of rational numbers in [0,1]. Then there exists a new finite set $\RR\,'$ of rational numbers in [0,1] such that $$\Gamma(\NN,\Phi)=\Phi(\RR\,').$$ Moreover, $\Gamma(\NN,\Phi)$ is a dcc set of rational numbers in [0,1] with only one accumulation point $1\in\Gamma(\NN,\Phi)$.
\end{lemma}

\begin{construction}\label{low-approximation}
\textup{Using the fact $\Gamma(\NN,\Phi)$ is a dcc set with only one accumulation point 1 from Lemma \ref{prop-hyperstandard}, for any $b\in[0,1]$, there exists a unique largest rational approximation $b_{\NN,\Phi}\leq b$ in $\Gamma(\NN,\Phi)$. For simplicity, in the following text we'll always take $\RR=\{0,1\}$ and denote this $b_{\NN,\Phi}$ by $b_{\NN}$. Similarly, for a boundary $B$ on $X$, we can construct its low approximation $B_{\NN}$.}
\end{construction}

The following lemma shows that this low approximation of the boundary $B$ preserves $n$-complements when $n\in\NN$.

\begin{lemma}\label{approximation-complements}
(\cite[Corollary 6]{Sh20}). Let $(X/Z,B)$ be a pair with a boundary $B$. If $(X/Z,B_{\NN})$ has an $n$-complement $(X/Z,B^+)$ with $n\in\NN$, then $(X/Z,B^+)$ is also an $n$-complement of $(X/Z,B)$. 
\end{lemma}

\subsection{Connectedness Lemma}

In this section we prove a modified version of the well-known Connectedness Lemma, which is Lemma \ref{connectedness-lemma}. Let's first recall the concept of locus of lc singularities in Definition \ref{LCS}.

\begin{definition}\label{LCS}
\textup{Let $X$ be a normal variety, $(X,D)$ a log pair. Then a subvariety $W\subset X$ is a} center of lc singularities, \textup{if there exists a birational contraction $f:Y\to X$ and a divisor $E$ (not necessarily exceptional) on $Y$, such that the discrepancy $a(E,D,X)\leq-1$ and $f(E)=W$. The union of all centers of lc singularities is called the} locus of lc singularities of $(X,D)$ \textup{and denoted by $\LCS(X,D)$.}
\end{definition}

%\begin{exam}\label{LCS-snc-case}
%\textup{Consider the special case for log surface pairs $(X,D)$, where $X$ is smooth and $D=\sum d_iD_i$ is a simple normal crossing divisor. Then explicitly, we have LCS$(X,D)=\cup_{i:d_i\geq1}D_i$.}
%\end{exam}

\begin{lemma}\label{connectedness-lemma}
(\cite[5.7]{Sh93}, \cite[Theorem 17.4]{KA92}). Let $(X/Z\ni o, D)$ be a pair with a proper local morphism $g:X\to Z\ni o$ and connected central fiber $g^{-1}(o)$, $D=\sum d_iD_i$ a divisor on $X$ such that:
\begin{enumerate}
\item if $d_i<0$, then $D_i$ is very $g$-exceptional in the sense that $g_*(\Oo_X(D_i))=\Oo_Z$;
\item $-(K_X+D)$ is big and nef over $Z$.
\end{enumerate}
%Let $h:Y\stackrel{f}{\longrightarrow}X\stackrel{g}{\longrightarrow}Z$ be a log resolution, $K_Y+\sum d_iE_i=f^*(K_X+D)$ the crepant pullback and $S=\sum_{i:d_i\geq 1}d_iE_i$. 
Then $\LCS(X,D)$ is connected in a neighborhood of the fiber $g^{-1}(o)$.
\end{lemma}

\begin{proof}
Let $X\to Z\,'\to Z$ be the Stein factorization of $g$. Since the central fiber $g^{-1}(o)$ is connected and $X\to Z\,'$ has connected fibers, the fiber of $Z\,'\to Z$ over $o$ is a single point $o\,'$. Then we can apply the well-known connectedness lemma to the contraction $X\to Z\,'$, and conclude that $\LCS(X,D)$ is connected near the fiber of $X\to Z\,'$ over $o\,'$, which is the same as $g^{-1}(o)$ set-theoretically. 
\end{proof}

The connectedness principle might not work for pairs $(X/Z\ni o, D)$ without the big assumption. However, we have a generalization of \ref{connectedness-lemma} for surface 0-pairs:

\begin{lemma}\label{LCS-connected-general}
(\cite[Theorem 6.9]{Sh93}). Let $X$ be a surface, $g:X\to Z\ni o$ a proper local morphism with connected central fiber $g^{-1}(o)$, $D$ an effective divisor on $X$ such that $K_X+D\equiv0$. If $\LCS(X,D)$ is disconnected in a neighborhood of $g^{-1}(o)$, then
\begin{enumerate}
\item $(X,D)$ is plt in a neighborhood of $g^{-1}(o)$, and
\item $\LCS(X,D)\cap g^{-1}(o)$ has two connected components.
\end{enumerate}
\end{lemma}

\begin{remark}\label{rem}
\textup{For a detailed discussion on Lemma \ref{LCS-connected-general} in higher dimensions and for pairs with $-(K_X+D)$ nef, see \cite[Theorem 1.2]{HH18} or \cite[Theorem 1.2]{Birkar22}.}
\end{remark}

\subsection{Inversion of Adjunction}

In this section we recall the fact that inversion of adjunction holds for log canonical surface pairs, which is Lemma \ref{inversion-adjunction-general} below.

\begin{lemma}\label{inversion-adjunction-general}
 Let $(X,S+D)$ be a log pair with $X$ a smooth surface, $S\subset X$ a smooth curve, and $D$ an effective $\Rr$-divisor on $X$ such that $S\not\subset\Supp D$. Then $(S,D|_S)$ is lc if and only if $(X,S+D)$ is lc in a neighborhood of $S$.   
\end{lemma}

The following is a special generalization of Lemma \ref{inversion-adjunction-general} which we'll use in the proof of Theorem \ref{big-nef-case} later.

\begin{corollary}\label{inversion-adjunction-special-case}
 Let $(X,S+D)$ be a log pair with $X$ a smooth surface, $S\subset X$ a smooth curve, $D=\sum d_iD_i$ an $\Rr$-divisor on $X$ such that $S\not\subset\Supp D$, and $D_i.S=0$ for all $i$ with $d_i<0$. Then if $(S,D|_S)$ is lc, $(X,S+D)$ is lc in a neighborhood of $S$.   
\end{corollary}

\begin{proof}
Since $D_{\geq0}|_S=D|_S$ by our assumption where $D_{\geq0}=\sum_{d_i>0}d_iD_i$, we can apply Lemma \ref{inversion-adjunction-general} to the pair $(X,S+D_{\geq0})$. Hence $(X,S+D_{\geq0})$ is lc in a neighborhood of $S$. Since $D\leq D_{\geq0}$, $(X,S+D)$ is also lc in a neighborhood of $S$.
\end{proof}

\begin{remark}\label{inversion-of-adjunction}
\textup{One can find a more general version of Lemma \ref{inversion-adjunction-general} in higher dimensions and for a wider range of varieties $X$ in \cite{Kaw06}. For inversion of adjunction in the log terminal case, see \cite[Theorem 5.50]{KM98}.}
\end{remark}

\subsection{Minimal Dlt Modification}

In some of our proofs, we need to lift lc pairs to a higher model with better singularities. That's where Minimal Dlt Modification of a lc pair as follows comes in handy.

\begin{proposition}\label{dlt-modification}
(\cite[Proposition 3.1.2]{Pr01}). Let $(X,D)$ be a lc normal surface pair where $D$ is a boundary. Then there exists a blow-up, namely the \textbf{minimal dlt modification} $f:Y\to X$ such that
\begin{enumerate}
\item $f^*(K_X+D)=K_Y+D_Y+\sum E_i$, where $D_Y$ is the birational transform of $D$ and $E_i$ are all the exceptional divisors over $X$;
\item $(Y,D_Y+\sum E_i)$ is dlt;
\item $Y$ is $\Qq$-factorial.
\end{enumerate}
\end{proposition} 

\subsection{Zariski Decomposition}

Traditionally Zariski decomposition are defined for effective divisors on nonsingular surfaces, on which the existence and uniqueness of Zariski decomposition is purely a linear algebraic result. There are generalizations of this classical concept to a wider range of algebraic varieties and divisors in higher dimensions, and here in this section we recall Zariski decomposition for pseudo-effective divisors on normal $\Qq$-factorial surfaces in the relative setting (\cite[Section 2]{Pr02}), where the 2-dimensional intersection theory works well. 

\begin{definition}\label{Zar-definition}
\textup{Let $X/Z$ be a relative projective normal $\Qq$-factorial surface, and $D$ be any divisor on $X$. Then a decomposition $D=M+F$ is called a} Zariski decomposition \textup{of $D$ if the following holds:
\begin{enumerate}
\item $F=\sum f_iD_i$, where $D_i$ are distinct vertical prime divisors over $Z$, and $f_i>0$;
\item $M$ is nef over $Z$;
\item the intersection matrix $(d_{ij})_{i\times j}$ is negative definite, where $d_{ij}=D_i.D_j$;
\item $M.D_i=0$ for all $i$.
\end{enumerate}
If such a decomposition exists and is unique, we usually call $M$ the \textit{mobile part} and $F$ the \textit{fixed part} of $D$.}
\end{definition}

\begin{remark}\label{high-dimension}
\textup{See \cite{Pr02} for a generalization of Zariski decomposition to higher dimensions.}
\end{remark}

It is well-known that Zariski decomposition exists and is unique for surface pairs:
\begin{theorem}\label{Zar}
(\cite[Theorem 2.2]{Pr02}). Let $X/Z$ be a projective relative normal $\Qq$-factorial surface and $D$ be a pseudo-effective divisor on $X$. Then there exists a unique Zariski decomposition $D=M+F$ of $D$.
\end{theorem}

\begin{proposition}\label{mobile-part-biggest-nef}
(\cite[Proposition 2.5]{Pr02}). Assumptions the same as in Theorem \ref{Zar} and let $L$ be a nef divisor over $Z$. If $L\leq D$, we have $L\leq M$.
\end{proposition} 

\begin{proof}
Let $M-L=(D-L)-F=M\,'-F\,'$, where $M\,'$ and $F\,'$ are effective divisors with no common components. Then to show $L\leq M$, it suffices to show $F\,'=0$. On one hand, since $D-L$ is effective, by our construction $F\,'\leq F$. Hence $M.F\,'=0$ by Definition \ref{Zar-definition}(4), and it follows that $(M-L).F\,'=-L.F\,'\leq0$ since $L$ is nef over $Z$. On the other hand, $(M-L).F\,'=(M\,'-F\,').F\,'\geq-(F\,')^2\geq 0$, where $M\,'.F\,'\geq0$ by construction, $(F\,')^2\leq0$ by Definition \ref{Zar-definition}(3) and the fact $F\,'\leq F$. Combining these, we have $(F\,')^2=0$, and $F\,'=0$ follows from Definition \ref{Zar-definition}(3).
\end{proof}

\begin{remark}\label{positive-mobile-part}
\textup{This property can be rephrased as: $M$ is the biggest nef divisor $L$ such that $L\leq D$. Another useful result from Proposition \ref{mobile-part-biggest-nef} is: if $D\,'\leq D$ are two pseudo-effective divisors on a surface $X$, then $M\,'\leq M$. Here $M\,'$ and $M$ are the mobile parts of $D\,'$ and $D$ respectively.}
\end{remark}

In general the mobile part $M$ is not necessarily semi-ample. But we can establish this useful property in the following special case.

\begin{lemma}\label{semi-ampleness}
Let $(X/Z,B)$ be a lc pair, where $X/Z$ is a projective relative normal surface, $B$ is effective and $D=K_X+B$ is pseudo-effective. Then the mobile part $M$ of $D$ is semi-ample over $Z$, and there exists a birational contraction $f:X/Z\to X\,'/Z$ such that the fixed part $F$ of $D$ is $f$-exceptional. Additionally, if $D$ is effective $M$ is also effective.
\end{lemma}

\begin{proof}
Run $(K_X+B)$-MMP on the pair $(X/Z,B)$, and suppose at each step we have an extremal contraction $f_i:(X_i,B_i)\to(X_{i+1},B_{i+1})$. Set $E_i:=K_{X_i}+B_i-f_i^*(K_{X_{i+1}}+B_{i+1})$, then $-E_i$ is $f_i$-nef and $f_i$-exceptional, hence by Negativity Lemma $E_i$ is effective. Let $X\,'/Z$ be the end result of this MMP, and $f:X/Z\to X\,'/Z$ be the composition of all $f_i$ with $B\,'=f_*B$, $D\,'=f_*D$, then $E:=K_X+B-f^*(K_{X\,'}+B\,')$ is also effective and $f$-exceptional. Note $D=K_X+B=f^*D\,'+E$ is a Zariski Decomposition of $D$, since $f^*D\,'$ is nef over $Z$, and $E\geq0$ is $f$-exceptional. By uniqueness we then have $M=f^*D\,'$ and $F=E$. Now $f$ is indeed a birational contraction for which $F$ is $f$-exceptional. The semi-ampleness of $D\,'$ over $Z$ follows from the relative abundance theorem in dimension 2: $(X\,',B\,')$ is lc, and $K_{X\,'}+B\,'=D\,'$ is nef over $Z$ by construction. Hence $M=f^*D\,'$ is also semi-ample over $Z$.

When $D$ is effective, we can apply Proposition \ref{mobile-part-biggest-nef} to our case with $L=0$, then we get $M$ is also effective.
\end{proof}

\subsection{0-contractions}

In this section we briefly review the definition of $0$-contractions. For more detailed definitions and properties, we refer to \cite[Section 7]{Sh20}.

\begin{definition}\label{0-contraction}
\textup{Suppose $f: X\to Z$ is a contraction of normal algebraic varieties, $B$ is an $\Rr$-divisor on $X$. Then in the usual sense $f$ is a \textit{$0$-contraction} if it satisfies (1) and (2) in the following:
\begin{enumerate}
\item $(X,B)$ is lc generically over $Z$;
\item $K_X+B\sim_{\Rr}0$ over $Z$;
\item the horizontal part $B^h$ of $B$ with respect to $f$ is generically a boundary over $Z$.
\end{enumerate}
But in this paper we always assume $0$-contractions also have property (3). }
\end{definition}

\begin{definition}\label{canonical-bundle-formula}
\textup{Under the same setting as in Definition \ref{0-contraction}, we have two divisors on $Z$ naturally related to $B$. One is called the \textit{divisorial part of adjunction} $B_{\dive}$ of $f$, which is uniquely determined by this $0$-contraction; the other is the \textit{moduli part of adjunction} $B_{\modu}$ of $f$, which is defined up to $\sim_{\Rr}$. For convenience we'll also use the b-divisor form $\Bb_{\modu}$ in this section. For explicit constructions of these two divisors, we refer to \cite[Section 7.2]{Sh20} or \cite[7.2]{PSh09}. The divisors $B_{\dive}$ and $B_{\modu}$ satisfy the following \textit{canonical bundle formula}: $$K_X+B\sim_{\Rr}f^*(K_Z+B_{\dive}+B_{\modu}).$$ }
\end{definition}

\begin{definition}
\textup{For a $0$-contraction $f: (X,B)\to Z$, we call a positive integer $I$ its \textit{adjunction index} if 
\begin{enumerate}
\item $\Kk_X+\BB\sim_If^*(\Kk_Z+\Bb_{\dive}+\Bb_{\modu})$, in particular, $K_X+B\sim_{I}f^*(K_Z+B_{\dive}+B_{\modu})$;
\item $I\Bb_{\modu}$ is b-Cartier and $\Bb_{\modu}$ is defined modulo $\sim_I$, in particular, $\Kk_X+\BB\sim_{I,Z}f^*(\Kk_Z+\Bb_{\dive})$.
\end{enumerate}
Actually, (2) implies (1). See \cite[7.3]{Sh20} for this result and more details.
}
\end{definition}

A well-known result about the moduli part $\Bb_{\modu}$ is the following theorem.

\begin{theorem}\label{nef-modular-part}
(\cite[Theorem 8]{Sh20}). Let $f:(X,B)\to Z$ be a $0$-contraction as in Definition \ref{0-contraction}, and $B$ is an effective $\Qq$-divisor generically over $Z$. Then $\Bb_{\modu}$ is a b-$\Qq$-divisor of $Z$, defined up to a $\Qq$-linear equivalence. Moreover, $\Bb_{\modu}$ is b-nef.
\end{theorem}

Proposition \ref{lifting-complements} below guarantees us that we can pull back $n$-components on the base $Z$ back to $X$ along a $0$-contraction $f:(X,B)\to Z$ if $n$ is divisible by the adjunction index $I$.

\begin{proposition}\label{lifting-complements}
(\cite[(10) in Section 7.5]{Sh20}). Suppose we have a $0$-contraction $f:(X,B)\to Z$ with adjunction index $I$, $X$ is a surface and $Z$ is a curve. If $n$ is an integer with $I|n$, then $(Z,B_{\dive}+B_{\modu})$ has an $n$-complement $\Rightarrow$ $(X,B)$ has an $n$-complement.
\end{proposition}

\subsection{Elliptic Fibrations}

We refer to \cite[Section 7.2]{Pr01} for a more complete theory on elliptic fibrations. Here we only list several definitions and propositions which are of importance to us later.

\begin{definition}\label{elliptic-fibration} 
\textup{An \textit{elliptic fibration} is a contraction $f:X\to Z$ from a surface to a curve such that its general fiber is a smooth elliptic curve. If we have $K_X\equiv0$ over $Z$ and $X$ is smooth additionally, we say $f$ is \textit{minimal}.}
\end{definition}

\begin{definition}\label{maximally-lc} 
\textup{Consider a pair $f:(X,B)\rightarrow Z\ni P$ from a surface to a curve. We say $(X/Z\ni P,B+pf^*P)$ is \textit{maximally lc} if $(X/Z\ni P,B+pf^*P)$ is lc and for any $q>p$, $(X/Z\ni P,B+qf^*P)$ is not lc.}
\end{definition}

Kodaira's classification of degeneration of elliptic fibers \cite{Kod63} lists all possible types of these singular fibers. We'll  recall this classical result in Theorem \ref{canonical-bundle-formula-elliptic} below, and restate two results in the view of canonical bundle formula and complements respectively.

\begin{theorem}\label{canonical-bundle-formula-elliptic}
(\cite[Theorem 7.2.10]{Pr01}). Let $f:X\to Z\ni P$ be a minimal elliptic fibration near $P$. By the canonical bundle formula we have $K_X=f^*(K_Z+D_{\dive})$ with effective $D_{\dive}=d_PP$. Then the following table lists all possible values of $d_P$ for possible types of fibers $f^*P$.
\begin{table}[h]
\centering
\begin{tabular}{|c|c|c|c|c|c|c|c|c|}
\hline
Type & $m$I$_n$ & II & III & IV & I$^*_b$ & II$^*$ & III$^*$ & IV$^*$ \\ \hline
$d_P$ & $1-\frac{1}{m}$ & $\frac{1}{6}$ & $\frac{1}{4}$ & $\frac{1}{3}$ & $\frac{1}{2}$ & $\frac{5}{6}$ & $\frac{3}{4}$ & $\frac{2}{3}$ \\
\hline
\end{tabular}
\end{table}
\end{theorem}  

\begin{proposition}\label{elliptic-fibration-complements}
(\cite[Theorem 3.1]{Sh97}). Assumptions the same as in Theorem \ref{canonical-bundle-formula-elliptic} and choose $p>0$ such that $(X/Z\ni P,pf^*P)$ is maximally lc. Then $K_X+pf^*P$ has index 1,2,6,4,3 respectively for the types $m$I$_b$, $m$I$^*_b$, II/II$^*$, III/III$^*$, IV/IV$^*$.
\end{proposition}  

\begin{proof}
By \cite[Theorem 3.1]{Sh97}, $(X/Z\ni P,pf^*P)$ has a regular $n$-complement $(X/Z\ni P,B^+)$ for $n=1,2,6,4,3$ in the corresponding cases. Since $K_X+pf^*P\equiv K_X+B^+\equiv0/Z\ni P$, we have $B^+=qf^*P$ for some $q\geq p$. Now $(X/Z\ni P,B^+=qf^*P)$ is lc and $(X/Z\ni P,pf^*P)$ is maximally lc, we have $q=p$. Hence $K_X+pf^*P=K_X+B^+$ has an index $n$. 
\end{proof}

\begin{proposition}\label{adjunction-index-elliptic-fibration}
(\cite[Example 7.16]{PSh09}). Let $f:(X, B)\to Z$ be a $0$-contraction with $X/Z$ an elliptic fibration. Then $f$ has an adjunction index $12$.
\end{proposition}  

\subsection{Known Boundedness Results}

Now we are ready to restate two known results for a certain class of  pairs $(X/Z,B)$ where $X/Z$ is of Fano type and $B$ has hyperstandard multiplicities. We'll use their special versions for surface pairs  in the proof of the main Theorem \ref{main} later. To begin with, we state the following result about the existence of adjunction index for such pairs:

\begin{theorem}\label{adjunction-index-existence-FT}
(\cite[Theorem 9]{Sh20}). Let $\RR$ be a finite set of rational numbers in $[0,1]$ and $d$ be a positive integer. Then there exists a positive integer $I=I(d,\RR)$ such that, for any contraction $f:(X,B)\to Z$ with
\begin{enumerate}
\item $f$ is a 0-contraction;
\item $X$ is of Fano type over $Z$ with $\dim X=d$;
\item the coefficients of $B^h$ are in $\Phi(\RR)$.  
\end{enumerate}
Then $f$ has an adjunction index $I$. 
\end{theorem}

\begin{corollary}\label{adjunction-index-existence}
Let $\RR$ be a finite set of rational numbers in $[0,1]$. Then there exists a natural number $I=I(\RR)$ such that, for any 0-contraction $f:(X,B)\to Z$ from a surface to a curve with coefficients of $B^h$ in $\Phi(\RR)$, $f$ has an adjunction index $I$. 
\end{corollary}

\begin{proof}
Let $F$ be a general fiber of $f$. By the adjunction formula $(K_X+B+F)|_F=K_F+B_F$. Taking degrees both sides, $0=(K_X+B+F).F\geq\deg K_F$, hence $F$ is either rational or elliptic.

\textit{Case 1}: In the case $f$ is a $\Pp^1$-fibration, we see $B_F\geq 0$ and hence $B^h\neq0$. We can then run a $K_X$-MMP over $Z$ and get $g:(X,B)\to(X\,', B\,'=g_*B)$. $(X\,',B\,')$ can't be a minimal model, otherwise $-B\,'\equiv K_{X\,'}$ is nef over $Z$, which contradicts that $B^h\neq0$. Hence we get a Mori fiber space $X\,'/Z$, and $X\,'$ is of Fano type over $Z$. We can then apply Theorem \ref{adjunction-index-existence-FT} to the 0-contraction $(X\,', B\,')\to Z$ ($B\,'=g_*B$ also has coefficients in $\Phi(\RR)$) and conclude that it has an adjunction index $I_1=I(\RR)$. Since $K_X+B=g^*(K_{X\,'}+B\,')$ is crepant, $I_1$ is also an adjunction index for $(X,B)\to Z$.

\textit{Case 2}: In the case $f$ is an elliptic fibration, $B^h=0$. Then Proposition \ref{adjunction-index-elliptic-fibration} concludes that $f$ has an adjunction index $I_2=12$. 

Combining the two cases above, we can take $I=I(\RR)=lcm(I_1,I_2)$.
\end{proof}

\begin{lemma}\label{coefficients-on-the-base}
(\cite[Theorem 5.5]{FMX19}, \cite[Proposition 6.3]{Birkar19}) Assumptions the same as in Corollary \ref{adjunction-index-existence}. Then there exists another finite set $\Ss\subset[0,1]$ of rational numbers depending only on $\RR$ such that the discriminant $B_{\dive}$ on $Z$ has coefficients in $\Phi(\Ss)$.
\end{lemma}

\begin{remark}\label{rem}
\textup{When the relative dimension of $f$ is 1, \cite[Theorem 8.1]{PSh09} generalizes this Corollary \ref{adjunction-index-existence} to all cases without the assumption $X/Z$ is of Fano type. Also see \cite[Theorem 5.5]{FMX19} for another version of this result.}
\end{remark}

Now let's recall the index conjecture for pairs $(X,B)$, which is established as Theorem \ref{index-theorem-surfaces} below in the case $X$ is a surface.

\begin{theorem}\label{boundedness-constraints-FT}
(\cite[Theorem 1.8]{Birkar19}). Let $\RR$ be a finite set of rational numbers in $[0,1]$ and $d$ be a positive integer. Then there exists a positive integer $n=n(d,\RR)$ such that, for any pair $(X/Z\ni o,B)$ and a contraction $f:(X,B)\to Z$ with
\begin{enumerate}
\item $(X,B)$ is lc of dimension $d$;
\item $X$ is of Fano type over $Z$;
\item the coefficients of $B$ are in $\Phi(\RR)$;
\item $-(K_X+B)$ is nef over $Z$.
\end{enumerate}
There is an $n$-complement $(X/Z\ni o,B^+)$ of $(X/Z\ni o,B)$ with $B^+\geq B$. Moreover, this complement is also an $mn$-complement for any positive integer $m$. 
\end{theorem}

\begin{theorem}\label{index-theorem-surfaces}
(\cite[Proposition 5.3]{FMX19}, \cite[Corollary 1.11]{PSh09}). Let $\RR$ be a finite set of rational numbers in $[0,1]$. Then there exists a positive integer $I=I(\RR)$ with the following. If $(X,B)$ is a surface pair such that 
\begin{enumerate}
\item $(X,B)$ is lc;
\item $K_X+B\sim_{\Rr}0$;
\item the coefficients of $B$ are in $\Phi(\RR)$.
\end{enumerate}
Then $I(K_X+B)\sim0$.
\end{theorem}

\begin{proof}
By taking a dlt modification (Proposition \ref{dlt-modification}), we can suppose $(X,B)$ is dlt. Then we have the following cases:

\textit{Case 1}: $B=0$ and $X$ has canonical singularities. We take a minimal resolution $g:Y\to X$. Then $K_Y-\sum a_iE_i=g^*K_X$ where $E_i$ are $g$-exceptional prime divisors. On one hand, $a_i\geq 0$ since $X$ has canonical singularities; on the other hand, since $K_Y$ is nef over $X$, by Negativity Lemma $a_i\leq 0$. Hence $a_i=0$ and $K_Y=g^*K_X$. We can then assume $X$ is smooth. In this case we can quote the classical result that $12K_X\sim0$ if $K_X\sim_{\Qq}0$ and set $I_1=12$.

\textit{Case 2}: $B=0$ and $X$ has klt singularities. By the classification of log terminal surface singularities \cite[Proposition 4.18]{KM98}, $X$ has quotient singularities and the existence of an index $n$ follows. A classical result from \cite{Bla95} has that $nK_X\sim0$ for some $n\leq21$ in this case. Hence for our result we can take $I_2=21!$.

\textit{Case 3}: $B\neq0$. Run $K_X$-MMP. Let $h:X\to X\,'$ be the composite of all contractions and $B\,'=h_*B$. Note $B\,'$ also has coefficients in $\Phi(\RR)$. If $X\,'$ is a minimal model, $K_{X\,'}$ is nef, $B\,'\geq0$ and $K_{X\,'}+B\,'\sim_{\Rr}0$. Hence $B\,'=0$ and by the previous case, $K_{X\,'}+B\,'$ has an index $I_2$. 

If $X\,'\to T$ is a Mori fiber space and $T$ is a curve, from our assumption we see that $f:(X\,',B\,')\to T$ is a 0-contraction, hence by Corollary \ref{adjunction-index-existence} $f$ has an adjunction index $J=J(\RR)$ and $$K_{X\,'}+B\,'\sim_J f^*(K_T+B\,'_T).$$ Here $B\,'_T=B\,'_{\dive}+B\,'_{\modu}$ is taken effective. By Lemma \ref{coefficients-on-the-base}, $B\,'_{\dive}$ has coefficients in some hyperstandard set $\Phi(\Ss)$ and since $JB\,'_{\modu}$ is integral, $B\,'_{\modu}$ has coefficients of the type $i/J$ for $i=0,1,...,J$. Hence by adding $i/J$, $i=1,2,...,J-1$ to $\Ss$ we get $\RR\,'$ and the coefficients of $B\,'_T$ are in the hyperstandard set $\Phi(\RR\,')$, which depends only on $\RR$. Note $\deg K_T\leq0$, $T$ is of Fano type and $K_T+B\,'_T\sim_{\Rr}0$. We can then apply Theorem \ref{boundedness-constraints-FT} to $(T,B\,'_T)$ to conclude that it has an index $N$, $N$ depends only on $\RR$. Then take $I_3:=\lcm(N,J)$, we have $I_3(K_{X\,'}+B\,')\sim f^*(I_3(K_T+B\,'_T))\sim0$, $I_3$ is an index for $K_{X\,'}+B\,'$, and $I_3$ depends only on $\RR$. 

If $T$ is a point, $X\,'$ is klt log del Pezzo and $\rho(X\,')=1$, in particular $X\,'$ has Fano type. Then by Theorem \ref{boundedness-constraints-FT}, $K_{X\,'}+B\,'$ has an index $I_4=I_4(\RR)$. 

Finally, since $K_X+B=h^*(K_{X\,'}+B\,')$ is crepant, an index $I$ for $K_{X\,'}+B\,'$ is also an index for $K_X+B$. We can then take our $I$ in the theorem to be the least common multiple of $I_j$, $j=1,2,3,4$. Notice that all $I_j$ depend only on $\RR$, hence so does $I$.
\end{proof}

\begin{remark}\label{rem}
\textup{As in the theorem, we call $I$ an $index$ for the 0-pair $(X/Z\ni o, B)$ if $I(K_X+B)\sim0$.}
\end{remark}

\begin{remark}\label{other-proofs}
\textup{It is known that the index conjecture holds true for non-klt 0-pairs $(X,B)$ when $\dim X\leq3$ (\cite[Corollary 1.12]{BSh10}). Using this result, we can avoid the use of Theorem \ref{boundedness-constraints-FT} in Case 3 of our proof.}
\end{remark}

\begin{remark}\label{index-higher-dimensions}
\textup{The index conjecture is widely expected to be true in arbitrary dimensions, especially in the case $(X,B)$ is not klt. The klt case is much more difficult and remains a mystery even in dimension 3. By Theorem \ref{boundedness-constraints-FT}, the index conjecture is true in the case $X$ is of Fano type.}
\end{remark}

\section{Boundedness for Curves}

\subsection{Classification of $\Rr$-complementary curves}

To find $n$-complements for a surfaces pair $(X,B)$, a natural thought is to lift $n$-complements of the restriction of the pair on some curves. Hence we need a boundedness result of $n$-complements for certain kinds of curves $C$. In view of Connectedness Lemma \ref{connectedness-lemma}, we focus on the case $C$ is a connected complete nodal curve. Proposition \ref{classification} below gives a complete classification of $\Rr$-complementary curve pairs $(C,B)$ in this case.

\begin{proposition}\label{classification}
Let $(C,B)$ be a pair, where $C$ is a connected complete nodal curve and $B$ is a boundary. If $(C,B)$ has an $\Rr$-complement, then one of the following occurs:
\begin{enumerate}
\item $C$ is a nonsingular rational curve and $\deg B\leq2$;
\item $C$ is a nonsingular elliptic curve and $B=0$;
\item $C$ is an irreducible rational curve with only one nodal singularity and $B=0$; 
\item $C=\sum\limits_{i=1}^{n}C_i$ is a cycle of smooth rational curves $C_i$, namely $C_i$ are the irreducible components and $C_i.C_j=1$ if $|i-j|=1$ or $\{i,j\}=\{1,n\}$, $C_i.C_j=0$ otherwise. And $B=0$;
\item $C=\sum\limits_{i=1}^{n}C_i$ is a chain of smooth rational curves $C_i$, namely $C_i$ are the irreducible components and $C_i.C_j=1$ if $|i-j|=1$, $C_i.C_j=0$ otherwise. If $B_i$ is the restriction of $B$ on $C_i$, then $\deg{B_1}\leq1$, $\deg{B_n}\leq1$ and $B_i=0$ otherwise.
\end{enumerate}
\end{proposition} 
 
We need a lemma before the proof.

\begin{lemma}\label{canonical-divisor}
Suppose $C$ is a connected complete curve with only nodal singularities and $C_i$ are the irreducible components. Consider the composition $\pi:\tilde{C}_i \to C_i\to C$ where $\tilde{C}_i\to C_i$ are the normalization of $C_i$, then we have $$K_{\tilde{C}_i}=\pi^*K_{C}-\sum P_j.$$ Here $P_j$ are all the points on $\tilde{C}_i$ mapped to singularities of $C$.
\end{lemma}

\begin{proof}
By the existence of an $\Rr$-complement $(C,B^+)$ we have $K_C+B^+\sim_{\Rr}0$ and $B^+\geq B\geq0$. Taking degree both sides, we have $\deg{K_C}\leq0$. Moreover, since $B^+$ is supported on the smooth locus of $C$, $B^+$ doesn't affect singularities and we can assume $(C,0)$ is lc.
  
Now assume $C$ is nonsingular. Then $\deg{K_C}=2g(C)-2\leq0$, $g(C)\leq 1$. If $g(C)=1$, $K_C\sim 0$, $B=B^+=0$ and we get case 1). If $C$ is rational, $K_C+B^+\sim_{\Rr}0$ is equivalent to $\deg{B^+}=2$. Existence of such $B^+$ is equivalent to $\deg B\leq 2$ and we get case 2).

If $C$ is singular and irreducible, consider its normalization $f:\tilde{C}\to C$. By Lemma \ref{canonical-divisor} we have $f^*K_C=K_{\tilde{C}}+\sum P_j$. Since $\deg{K_C}\leq0$, $\deg{(K_{\tilde{C}}+\sum P_j)}=2g(\tilde{C})-2+l\leq 0$. Here $l$ is the number of points on $\tilde{C}$ mapped to singularities of $C$. Suppose $C$ has $r$ nodal singularities, then $l=2r\leq2$, $r=1$ and $\tilde{C}$ is rational. As a result $K_C\sim 0$, $B=B^+=0$ and this is case 3).

For the rest of cases $C$ is reducible. For $\pi_i:\tilde{C}_i \to C$ defined in Lemma \ref{canonical-divisor}, we have $\deg{(K_{\tilde{C}_i}+\sum P_j)}=2g(\tilde{C}_i)-2+l_i\leq 0$, where $l_i$ is the number of points on $\tilde{C}_i$ mapped to singularities of $C$. There are at least one singularity of $C$ lying on $C_i$ by the connectedness assumption, so $l_i \geq1$. Hence $g(\tilde{C}_i)=0$ and $l_i\leq2$. If $C_i$ has a nodal singularity itself, we'll get $l_i\geq3$ and contradiction. Hence $C_i$ are rational and smooth for all $i$. Now $l_i\leq2$ means every $C_i$ has at most 2 intersections with other $C_j$. Starting with $C_1$, we can rearrange $C_i$ such that $C_i$ intersects $C_{i+1}$ for $1\leq i\leq n-1$. If $C_n$ intersects $C_1$, then $l_i=2$ for all $i$, and $K_{C_i}+\sum P_j+B^+_i\sim0$ implies $B_i=B^+_i=0$, hence $B=0$. This is case 4). Otherwise $C_1$ doesn't meet $C_n$, $l_1=l_n=1$ and $l_i=0$ for $2\leq i \leq n-1$. By the same reasoning $B_i=0$ for $2\leq i \leq n-1$. For $i=1$ or $n$, $K_{C_i}+\sum P_j+B^+_i\sim0$ is equivalent to $\deg{B^+_i}=1$, so $\deg{B_i}\leq1$. This is case 5). 
\end{proof}

\subsection{Boundedness  of $n$-complements on curves}

Now we are ready to prove the boundedness of $n$-complements under complementary restrictions, which is Corollary \ref{boundedness-curve}. Corollary \ref{boundedness-curve} follows from Theorem \ref{simultaneous} below about simultaneous construction of $n$-complements for curves. Before we start, We'll need some lemmas and notations. All norms throughout are maximal absolute value norm.

\begin{notation}\label{note}
For a real number $a\in[0,1]$ and a positive integer $n$, we introduce a real number $a^{[n]}$ as in Definition \ref{n-comp} of $n$-complements: $a^{[n]}=1$ if $a=1$, and $a^{[n]}=\frac{\lf (n+1)a\rf}{n}$ otherwise. Note $a^{[n]}\leq1$ always holds.
\end{notation}

\begin{lemma}\label{addition}
For arbitrary n real numbers $a_1,a_2,...a_n$, $\sum\lf a_i\rf\leq\lf\sum a_i\rf$.
\end{lemma}

\begin{lemma}\label{continuity}
Suppose we have a continuous map $\lambda:W\to V$ between two finite linear spaces with norms, a vector $e\in W$ such that $\lambda(e)\neq0$ and a real number $\epsilon>0$. Then we can find a $\delta>0$ such that for every $e\,'$ with $||e-e\,'||<\delta$, we have $\lambda(e\,')\neq0$ and $||\frac{\lambda(e)}{||\lambda(e)||}-\frac{\lambda(e\,')}{||\lambda(e\,')||}||<\epsilon$.
\end{lemma}

\begin{theorem}\label{simultaneous}
Suppose we have $I,v,e,\epsilon$ as in Definition \ref{comp-restriction}, a positive integer $d$ and $m$ positive rational numbers $c_1,c_2,...,c_m$. Then there exists a finite set $\NN=\NN(I,v,e,\epsilon,d,c_1,...,c_m)$ such that the following holds:
\begin{enumerate}
\item $\NN$ consists of positive integers under complementary restrictions with respect to $I,v,e,\epsilon$;
\item For $m$ vectors $B_i=(b_{i1},...,b_{id})\in\Rr^d$ satisfying $\sum_{j=1}^d b_{ij}\leq c_i$ and $b_{ij}\in[0,1]$ for $i=1,2,...,m$, $j=1,2,...,d$, there exists an $n\in\NN$ such that $\sum^d_{j=1}b_{ij}^{[n]}\leq c_i$ for $i=1,2,...,m$. 
\end{enumerate}
\end{theorem} 

\begin{proof}
Consider the following subset of $\Rr^{md}$: $$S:=\{(B_1,...,B_m)|\sum^d_{j=1}b_{ij}\leq c_i, b_{ij}\in [0,1], i=1,2,...,m, j=1,2,...,d\}.$$ $S$ is then compact. For every point $B$ in $S$ we will define an open neighborhood $U_B$ of $B$ and an integer $n_B$ satisfying the complementary restrictions, such that for any point $(A_1,A_2,...,A_m)\in U_B$ and this $n_B$, $\sum^d_{j=1}a_{ij}^{[n_B]}\leq c_i$ for $i=1,2,...,m$. The sets $U_B$ form an open cover of $S$, by compactness it has a finite subcover $\cup U_{B_q}$. Then $\NN=\{n_{B_q}\}$ is our desired set.

We fix a point $B\in S$. Without loss of generality we can assume $b_{1j}<1$ if $j\leq r_1$ and $b_{1j}=1$ otherwise. Then we set the positive constants $\delta_1=\min_{j\leq r_1}\{1-b_{1j},1/(2r_1)\}$ and $\delta_i$ similarly; in the special case $r_i=0$ let $\delta_i=\infty$. Define $\lambda:\Rr^{l+md}\to\Rr^l$, $(b_1,b_2,...,b_{l+md})\mapsto (b_1,b_2,...,b_l)$ to be the projection onto first $l$ entries. Let $u=(v,B)$ and choose some direction $e\,'$ of the rational affine span $\la u\ra$ such that $\lambda(e\,')=e$. For this $\lambda$, $e\,'$ and $\epsilon$ we have a $\delta$ as in Lemma \ref{continuity}. Since we can make this $\delta$ arbitrarily small, in the following we'll fix a $\delta<\min_{1\leq i\leq m}\{\epsilon,\delta_i\}$. Now we can choose some $n$ under complementary restrictions with respect to $I\,',u,e\,',\delta$ where $I\,'=lcm(I,q_i)$ and $q_i$ are the denominators of $c_i$. Then we have an approximation $u_n=(v_n,B_n)$ ($v_n=\lambda(u_n)$) with denominator $n$ of $u$ such that $||u_n-u||<\delta/n$ and $||\frac{u_n-u}{||u_n-u||}-e\,'||<\delta$.  

We claim $n=n_B$ is an index satisfying complementary restrictions with respect to $I,v,e,\epsilon$. Divisibility $I|n$ is clear since $I|I\,'$. Since $\lambda(e\,')=e\neq0$, we can apply Lemma \ref{continuity} to $\lambda$, $e\,'$ and $\frac{u_n-u}{||u_n-u||}$. As a result $||\frac{v_n-v}{||v_n-v||}-e||<\epsilon$, which is 4) in Definition \ref{comp-restriction}. From $||u_n-u||<\delta/n<\epsilon/n$, we also have $||v_n-v||<\epsilon/n$. This is 3) in Definition \ref{comp-restriction}. So $n$ satisfies 1) in the theorem indeed. As an addendum we have an approximation $||B_n-B||<\delta/n$ of $B$ with denominator $n$.

Let $B_n=((B_1)_n,...,(B_m)_n)$ be this approximation and we'll construct the neighborhood $U_B$. We focus on $B_1$ and $(B_1)_n$ for now. Let $(B_1)_n=(\frac{M_1}{n},...,\frac{M_d}{n})$ where $M_j$ are integers. Then set $$U_1=\{(a_1,..,a_d)\Big||a_j-\frac{M_j}{n}|<\frac{\delta}{n},a_j<1-\delta, j\leq r_1;a_j>1-\frac{1}{2n(d-r_1)},j>r_1;\sum a_j\leq c_1\}.$$ One can check $B_1$ is indeed in $U_1$: for $j\leq r_1$, $||(B_1)_n-B_1||<\delta/n$ implies $|b_{1j}-\frac{M_j}{n}|<\frac{\delta}{n}$, and $b_{1j}\leq1-\delta_1<1-\delta$ by our choice of $\delta$; for $j>r_1$ this is clear since $b_{1j}=1$. 

For any $(a_1,...,a_d)$ in $U_1$, let $\epsilon_j=na_j-M_j$. For $j\leq r_1$, by construction $|\epsilon_j|<\delta$ and $a_j<1$, so we have $$a^{[n_j]}=\frac{\lf (n+1)a_j\rf}{n}=\frac{\lf M_j+\epsilon_j+a_j\rf}{n}=\frac{\lf \epsilon_j+a_j \rf}{n}+a_j-\frac{\epsilon_j}{n}<a_j+\frac{\delta}{n}.$$ Note $\epsilon_j+a_j<\delta+1-\delta=1$, so $\lf \epsilon_j+a_j\rf\leq0$. Summing these up for $j$ from $1$ to $r_1$ and note $r_1\delta<r_1\delta_1\leq1/2$ we get
\begin{equation}
\sum_{j\leq r_1}a^{[n]}_j<\sum_{j\leq r_1} a_j+\frac{r_1\delta}{n}<\sum_{j\leq r_1} a_j+\frac{1}{2n}
\end{equation} 
For $j>r_1$ part we have an estimate
\begin{equation}
\sum_{j>r_1}a^{[n]}_j\leq\sum_{j>r_1}1<\sum_{j>r_1}(a_j+\frac{1}{2n(d-r_1)})=\sum_{j>r_1}a_j+\frac{1}{2n}
\end{equation}
Adding (1) and (2) we have $\sum a^{[n]}_j<\sum a_j +\frac{1}{n}\leq c_1+\frac{1}{n}$; in the special cases $r_1=0$ and $r_1=d$, we don't have (1) or (2), so we have actually a stronger inequality $\sum a^{[n]}_j<c_1+\frac{1}{2n}$. This implies $\sum a^{[n]}_j\leq c_1$ since the denominator of $c_1$ divides $n$ and $\sum a^{[n]}_j$ has denominator $n$.

We construct open sets $U_i$ similarly. Then the set $U_B=(U_1\times U_2\times...\times U_m)\cap S$ is an open neighborhood of $B$ meeting all of our requirements.
\end{proof}

\begin{remark}\label{rem}
\textup{In the proof we only used the assumption $b_{ij}\geq0$ for the compactness of $S$. So we can drop this condition if the set $S$ of boundaries is already compact. For example, this is the case when we consider all subboundaries $B_i$ with $\sum_{j=1}^d b_{ij}=c_i$ and $b_{ij}\leq1$.}
\end{remark}

\begin{corollary}\label{boundedness-curve}
Suppose we have the data $I,v,e,\epsilon$ as in Definition \ref{comp-restriction}. Then we can find a finite set $\NN=\NN(I,v,e,\epsilon)$ such that:
\begin{enumerate}
\item {\rm Restrictions:} $\NN$ consists of integers under complementary restrictions;
\item {\rm Existence of $n$-complements:} suppose we have a curve pair $(C,B)$ where $C$ is a connected nodal curve, $B$ is a boundary supported on the smooth locus of $C$, and $-(K+B)$ is nef on every irreducible component of $C$. Then $(C,B)$ has an $n$-complement $(C,B^+)$ for some $n\in\NN$, where $B^+$ has the same support as $B$.
\end{enumerate}
\end{corollary}

\begin{proof}
By Proposition \ref{classification}, the only interesting cases are 3) and 5), since in other cases $B=0$, and $(C,0)$ is an $n$-complement of itself for any $n$. For case 3), we can join small multiplicities of $B$ together by Lemma \ref{addition} until there are at most one of them $<1/2$. As a result we can assume there are at most 4 nonzero multiplicities in $B$. Applying Theorem \ref{simultaneous} to the case $m=1$, $c_1=2$, $d=4$, we get a finite set $\NN_1$. We can then find $n\in\NN_1$ such that $B^{[n]}=\sum b_i^{[n]}P_i$ has degree $\leq2$, and some boundary $B^+\geq B^{[n]}$ with degree 2. Then $(C,B^+)$ is an $n$-complement of $(C,B)$ by definition. 

For case 5), similarly we can assume there are at most 2 nonzero multiplicities in $B_1$ and $B_h$. Applying Theorem \ref{simultaneous} to $m=2$, $c_1=c_2=1$ and $d=2$, we get another finite set $\NN_2$. We can then find $n\in\NN_2$ such that $B_i^{[n]}=\sum b_{ij}^{[n]}P_{ij}$ has degree $\leq1$, and some boundary $B_i^+\geq B_i^{[n]}$ with degree 1 for $i=1,h$. Then $(C,B^+=B^+_1+B^+_h)$ is an $n$-complement of $(C,B)$.

Now $\NN=\NN_1\cup\NN_2$ is our desired set.
\end{proof}

\section{A special big case of Theorem \ref{main}}

In this chapter, using our boundedness result Corollary \ref{boundedness-curve} for curves, Kawamata-Viehweg vanishing and Connectedness Lemma \ref{connectedness-lemma}, we show that Theorem \ref{main} holds when $-(K_X+B)$ is big and nef, which is Corollary \ref{big-case}. This case is crucial in the proof of Theorem \ref{main} later. We'll need the following preliminary results.

\begin{notation}\label{note}
\rm{For any divisor $D$, we have the unique decomposition of $D=D_{\geq 0}-D_{\leq 0}$ into its positive and negative part, where $D_{\geq 0}$ and $D_{\leq 0}$ are two effective divisors with no common components.}
\end{notation}

\begin{lemma}\label{negativity-curve}
Let $X$ be a smooth surface, $S\subset X$ a nonsingular curve and $D$ a divisor with $S\not\subset\Supp D$. If $-(K+S+D)$ is nef on the support of $D_{\leq 0}$ and $D_{\leq 0}$ is contractible, then $S$ has no intersection with components of $D_{\leq 0}$. 
\end{lemma}

\begin{proof}
If $D_{\leq 0}=0$, the result is clear. We assume $D_{\leq 0}\neq0$ and use induction on the number of components of $D_{\leq 0}$. Since $D_{\leq 0}$ is contractible, by negative definiteness we have $(D_{\leq 0})^2<0$. Hence there exists an irreducible curve $E\subset\Supp D_{\leq 0}$ such that $D_{\leq 0}.E<0$. It follows that $$(K+S).E\leq(K+S+D_{\geq 0}).E\leq D_{\leq 0}.E<0.$$ Since $E$ is contractible, we also have $E^2<0$. Hence $E$ is a $(-1)$-curve and $S.E=0$.

Now let $g:X\to Y$ be the contraction of $E$ and consider $g(S)$ and $g(D)$ on $Y$. Let's check that $K_Y+g(S)+g(D)=g_*(K+S+D)$ also satisfies all the assumptions of the lemma on $Y$. For an irreducible curve $C\subset\Supp(g(D)_{\leq 0})$ we have $$g_*(K+S+D).C=(K+S+D).g^*C=(K+S+D).(\tilde{C}+aE)\leq0.$$ Here the total transform $g^*C$ consists of two parts: $\tilde{C}$, the birational transform of $C$ on $X$ and the exceptional part $aE$ with $a\geq0$. They are both in the support of $D_{\leq0}$, hence $(K+S+D).\tilde{C}\leq0$ and $(K+S+D).E\leq0$ by assumption. It's then established $-(K_Y+g(S)+g(D))$ is nef on $\Supp g(D)_{\leq 0}$. It's also clear that $g(D)_{\leq 0}=g(D_{\leq 0})$ is contractible and has strictly smaller number of components than $D_{\leq 0}$. Hence by induction hypothesis, $g(S)$ has no intersection with $\Supp g(D)_{\leq 0}$. Since $S$ is disjoint from $E$ and $\Supp D_{\leq 0}=\Supp g(D)_{\leq 0}\cup E$, we get our result. 
\end{proof}

\begin{corollary}\label{negativity-curves}
Notations and assumptions the same as in Lemma \ref{negativity-curve}. The result still holds if $S$ is a sum of nonsingular curves.
\end{corollary} 

\begin{lemma}\label{Q-Cartier}
Let $(X,B)$ be a lc log surface pair, $D\subset\Supp B$ is a prime divisor. Then $D$ is $\Qq$-Cartier. 
\end{lemma}

\begin{proof}
Let $f:Y\to X$ be a log resolution of $(X,B)$, and $K_Y+B_Y=f^*(K_X+B)+\sum a_iE_i$. Here $B_Y$ is the birational transform of $B$, $E_i$ are the exceptional divisors. Choose a $0<\epsilon\ll1$, then $f$ is also a log resolution of the pair $(X,B-\epsilon D)$. Let $K_Y+B_Y-\epsilon D_Y\equiv_f \sum a\,'_iE_i$, where $D_Y$ is the birational transform of $D$. Then $\sum (a_i-a\,'_i)E_i\equiv_f\epsilon D_Y$ is $f$-nef. By negativity lemma, $a\,'_i\geq a_i\geq -1$ since $K_X+B$ is lc. Hence $(X,B-\epsilon D)$ is numerically lc, which is the same as lc on surfaces. It follows that $K_X+B-\epsilon D$ is $\Qq$-Cartier and $D$ is also $\Qq$-Cartier.
\end{proof}

\begin{theorem}\label{big-nef-case}
Fix the data $I,v,e,\epsilon$ as in Definition \ref{comp-restriction}. Let $(X/Z\ni o,B)$ be a log surface pair with a proper local morphism $g:X\to Z\ni o$ and a connected central fiber $X_o$. Assume further
\begin{enumerate}
\item $-(K_X+B)$ is big and nef over $Z$;
\item $B$ is a boundary, and $(X,B)$ is lc.
\end{enumerate}
Then there exists a finite set $\NN=\NN(I,v,e,\epsilon)$ consisting of integers under complementary restrictions such that $(X/Z\ni o,B)$ has an $n$-complement for some $n\in\NN$.
\end{theorem} 

\begin{proof}
We claim $\NN=\NN_1\cup\NN_2$ is one desired finite set. Here we take $\NN_1$ to be the set $\NN$ as in \cite[Theorem 3]{Sh20} for 2-dimensional case. $\NN_2$ is the set $\NN(I,v,e,\epsilon)$ from Corollary \ref{boundedness-curve}.

\textbf{Step 1}: In this step we'll first reduce to the case when there are no horizontal divisors in the reduced center $\lf B\rf$. Suppose the theorem already holds for $\NN$ when horizontal parts of $\lf B\rf$ are empty. We'll use induction on the number of horizontal components in $\lf B\rf$. If there is any prime horizontal divisor $D$ in $\lf B\rf$, we decrease its multiplicity by $\epsilon<\frac{1}{N+1}$ where $N$ is the biggest number in $\NN$. Explicitly, we set $B\,'=B-\epsilon D$, then the number of horizontal components in $\lf B\,'\rf$ drops by 1.

For this new pair $(X/Z\ni o,B\,')$, we check it still satisfies the assumptions in the theorem. By Lemma \ref{Q-Cartier} $K_X+B\,'$ is again $\Rr$-Cartier and $B\,'$ is also a boundary. For (1), since $-(K_X+B)$ is big and nef over $Z$, and $\epsilon D$ is horizontal, hence big and nef over $Z$, the sum $-(K_X+B\,')=-(K_X+B)+\epsilon D$ is also big and nef over $Z$. For (2), it's clear since $(X,B)$ is lc and $B-B\,'=\epsilon D$ is effective and $\Qq$-Cartier. 

Now by induction hypothesis $(X/Z\ni o,B\,')$ has an $n$-complement $(X/Z\ni o,B^+)$ for some $n\in\NN$. We claim $(X/Z\ni o,B^+)$ is also an $n$-complement for $(X/Z\ni o,B)$. For this, we only need to check the multiplicities $d$, $d\,'$ and $d^+$ of $D$ in $B$, $B\,'$ and $B^+$ respectively. We have $d=1$, $d\,'=1-\epsilon$ and $d^+\geq\lf(n+1)(1-\epsilon)\rf/n$. Note $(n+1)(1-\epsilon)\geq n$ by our choice of $\epsilon$, $d^+=1$. Hence by definition $(X/Z\ni o,B^+)$ is also an $n$-complement of $(X/Z\ni o,B)$. 

\textbf{Step 2}: In this step we consider the case when $(X,B)$ is klt. By assumption (1), $X/Z\ni o$ is then of Fano type, hence $(X/Z\ni o,B)$ has a klt $\Rr$-complement. Then by \cite[Theorem 3]{Sh20} $(X/Z\ni o,B)$ has an $n$-complement for some $n\in\NN_1$. 

\textbf{Step 3}: In this step we assume $(X,B)$ is lc but not klt. We'll construct an $n$-complement $(X/Z\ni o,B^+)$ of $(X/Z\ni o, B)$ for some $n\in\NN_2$ in this case.

Let $h:Y\stackrel{f}{\longrightarrow}X\stackrel{g}{\longrightarrow}Z$ be a log resolution of $(X,B)$, and $f^*(K_X+B)=K_Y+\sum d_jD_j$ be the crepant pullback. Define $S:=\sum_{j:d_j=1}D_j$ and $A:=\sum_{j:d_j<1}d_jD_j$. Now we check the conditions from Corollary \ref{boundedness-curve} on the pair $(S,A|_S)$. Firstly, $S\neq\emptyset$ since $(X,B)$ is not klt, $S$ is connected near the central fiber $h^{-1}(o)$ by Lemma \ref{connectedness-lemma}, and $S$ is snc because $f$ is a log resolution of $(X,B)$. Hence $S$ is a nonempty connected nodal curve near $h^{-1}(o)$. Secondly, $A_{\leq0}$ only consists of $f$-exceptional divisors since $B$ is effective, hence $A_{\leq0}$ is contractible. Moreover, $-(K_Y+S+A)=-f^*(K_X+B)$ is nef over $Z$ by assumption 1), and in particular $-(K_Y+S+A)$ is nef on the support of $A_{\leq0}$. Hence by Corollary \ref{negativity-curves}, $S$ has no intersection with components of $A_{\leq0}$, therefore $A|_S$ is a boundary divisor of $S$. Since $S+A$ is snc, $A|_S$ is supported on the smooth locus of $S$. Finally, $-(K_S+A|_S)$ is nef on every irreducible component of $S$, since $-(K_Y+S+A)$ is nef over $Z$ and $S$ is vertical by our reduction in Step 1. Hence we can apply Corollary \ref{boundedness-curve} to the pair $(S,A|_S)$ and get an $n$-complement $(S,(A|_S)^+)$ for some $n\in\NN_2$, where $(A|_S)^+$ has the same support as $A|_S$.

Now $-(K_Y+S+A)$ is big and nef over $Z$, and $S+A$ is snc since $f$ is a log resolution of $(X,B)$, by Kawamata-Viehweg vanishing we have 
\begin{eqnarray*}
R^1h_*(\Oo_Y(-nK_Y-(n+1)S-\lf(n+1)A\rf))=\\
R^1h_*(\Oo_Y(K_Y+\lceil-(n+1)(K_Y+S+A)\rceil))=0.
\end{eqnarray*}
This implies the surjection
\begin{eqnarray*}
H^0(Y,\Oo_Y(-nK_Y-nS-\lf(n+1)A\rf))\twoheadrightarrow \\
H^0(S,\Oo_S(-nK_S-\lf(n+1)A\rf|_S)).
\end{eqnarray*}
Since $(S,(A|_S)^+)$ is an $n$-complement of $(S,A|_S)$ and all multiplicities in $A|_S$ are $<1$ by construction, $n(A|_S)^+-\lf(n+1)A|_S\rf\in|-nK_S-\lf(n+1)A|_S\rf|$. From the surjectivity, there exists $\overline{A}\in|-nK_Y-nS-\lf(n+1)A\rf|$ which restircts to $n(A|_S)^+-\lf(n+1)A|_S\rf$. Set $A^+=\frac{1}{n}(\overline{A}+\lf(n+1)A\rf)$, then $A^+|_S=(A|_S)^+$ and $n(K_Y+S+A^+)\sim0$.

\textbf{Step 4}: In this step we show $(X/Z\ni o,B^+)$ is an $n$-complement of $(X/Z\ni o,B)$, where $B^+=f_*(S+A^)$. Since $n(K_Y+S+A^+)\sim0$, we have $n(K_X+B^+)\sim0$. Moreover, $A^+=\frac{1}{n}(\overline{A}+\lf(n+1)A\rf)\geq\frac{1}{n}\lf(n+1)A\rf$ since $\overline{A}\geq0$. Combining this with the fact that coefficients of $A$ are $<1$, $(X/Z\ni o,B^+)$ satisfies condition 1) in Definition \ref{n-comp}. Now it remains to check that $(X,B^+)$ is lc near the fiber over $o$. 

Firstly, we show that $(Y,S+A^+)$ is lc near $S$. For this, we apply inversion of adjunction \ref{inversion-adjunction-special-case} on every irreducible component $S_i$ of $S$: $Y$ and $S_i$ are both smooth since $f$ is a log resolution, and it's also clear $S_i\not\subset\Supp(S-S_i+A^+)$. Moreover, since $\Supp A^+_{\leq0}\subset\Supp A_{\leq0}$ and $S$ has no intersection with components of $A_{\leq0}$, $S$ also has no intersection with components of $A^+_{\leq0}$. Hence it remains to check that the restriction $(S-S_i+A^+)|_{S_i}$ is a boundary on $S_i$. Since $A^+|_S=(A|_S)^+$ and is supported on the smooth locus of $S$, $A^+|_S=\sum\limits_{i}A^+|_{S_i}$ is the sum of its disjoint restrictions on $S_i$. Moreover, $A^+|_{S_i}$ and $(S-S_i)|_{S_i}$ have no common components.  Hence it suffices to show $(S-S_i)|_{S_i}$ and $A^+|_{S_i}$ are both boundaries. On one hand, $(S-S_i)|_{S_i}$ is a boundary since $S=\sum\limits_{i}S_i$ is snc. On the other hand, $A^+|_S$ is also a boundary since $(S,A^+|_S=(A|_S)^+)$ is an $n$-complement of $(S,A|_S)$ with $A|_S$ a boundary. In particular $A^+|_{S_i}$ are boundaries for every $i$.

Secondly, we prove a stronger result that $(Y,S+A^+)$ is lc near any fiber of $f:Y\to X$ that intersects $S$. Since $A^+=\frac{1}{n}(\overline{A}+\lf(n+1)A\rf)$ and $\overline{A}$ is effective, components in $A^+$ with negative coefficients are contracted, assumption 1) is satisfied in Corollary \ref{connectedness-lemma}. Assumption 2) also holds since $K_Y+S+A^+\sim 0$ and $f$ is birational. Hence we can apply connectedness lemma \ref{connectedness-lemma} to $(Y/X,S+A^+)$, and conclude that $\LCS(Y,S+A^+)$ is connected in a neighborhood of any fiber of $f$ that intersects $S$. Now assume $(Y,S+A^+)$ is not lc near one of such fiber $F$. Then from the facts that $(Y,S+A^+)$ is lc but not klt near $S$, $F\cap S\neq\emptyset$ and $\LCS(Y,S+A^+)$ is connected near $F$, we can find an irreducible curve $T$ in the fiber $F$, such that $T$ is contained in Supp$(S+A^+)$ with multiplicity $1$, $T$ passes through a point where $(Y,S+A^+)$ is not lc and a point where $(Y,S+A^+)$ is lc. Consider the different $A\,'=(S+A^+-T)|_T$, then $A\,'$ is effective by Lemma \ref{negativity-curve}. On one hand, $T$ contains a point where $(Y,S+A^+)$ is not lc. Hence by Lemma \ref{inversion-adjunction-general}, $(T,A\,')$ is not lc, and this leads to a point on $A\,'$ with multiplicity $>1$. On the other hand, the point where $(Y,S+A^+)$ is lc leads to a different point on $A\,'$ with multiplicity $\geq1$. Hence $\deg A\,'>2$ and $\deg(K_T+A\,')>0$. This contradicts the fact that $K_T+A\,'=(K_Y+S+A^+)|_T$ is antinef over $X$. Therefore, $(Y,S+A^+)$ is lc near any fiber of $f$ that intersects $S$. Moreover, $f^*(K_X+B^+)=K_Y+S+A^+$ is crepant since $n(K_Y+S+A^+)\sim0$. Hence by pushing forward to $X$, we conclude that $(X,B^+)$ is lc near $f(S)$.

Finally, if $\LCS(X,B^+)$ is disconnected in a neighorhood of the fiber over $o$, we can apply Lemma \ref{LCS-connected-general} to the pair $(X,B^+)$ since $K_X+B^+\sim0$. We then conclude that $(X,B^+)$ is plt in a neighborhood of $g^{-1}(o)$, hence lc. Otherwise, in a neighborhood of the fiber over $o$, the non-klt center $\LCS(X,B^+)$ is connected and contains $f(S)$. If $(X,B^+)$ is not lc near the fiber over $o$, $\LCS(X,B^+)$ also contains a point $P\notin f(S)$ where $(X,B^+)$ is not lc. Now by the same reasoning as in the paragraph above, we get a contradiction. Hence $(X,B^+)$ is also lc near the fiber over $o$ in this case.

\end{proof}

Now we can prove the following special case of the main theorem \ref{main}.
\begin{corollary}\label{big-case}
Suppose we have the data $I,v,e,\epsilon$ as in Definition \ref{comp-restriction}. Then we can find a finite set $\NN=\NN(I,v,e,\epsilon)$ such that:
\begin{enumerate}
\item {\rm Restrictions:} $\NN$ consists of integers under complementary restrictions;
\item {\rm Existence of $n$-complements:} suppose we have a log surface pair $(X/Z\ni o,B)$ with a proper local morphism $X/Z\ni o$, connected central fiber $X_o$ and an $\Rr$-boundary $B$. Then if $-(K_X+B)$ is big over $Z$, and $(X/Z\ni o,B)$ has an $\Rr$-complement, it also has an $n$-complement for some $n\in\NN$.
\end{enumerate}
\end{corollary} 

\begin{proof}
Just take $\NN$ to be the same set from Theorem \ref{big-nef-case}. By assumption $(X/Z\ni o,B)$ has an $\Rr$-complement $(X/Z\ni o,B+B\,')$ where $B\,'\sim-(K_X+B)$ is big over $Z$ and effective. Then we can take the Zariski decomposition $B\,'=M+F$ of $B\,'$ by Theorem \ref{Zar}. For the pair $(X/Z\ni o,B+F)$, the conditions (1) and (2) in Theorem \ref{big-nef-case} hold: for (1), $-(K_X+B+F)\sim_{\Rr}M/Z\ni o$ is big over $Z$ since $M$ is the mobile part of the big divisor $B\,'$ over $Z$, and $M$ is nef over $Z$ by (2) in Definition \ref{Zar-definition}; for (2), notice that $(X,B+B\,')=(X,B+F+M)$ is lc, $M$ is effective by Lemma \ref{semi-ampleness} and $Qq$-Cartier by Lemma \ref{Q-Cartier}, hence $(X,B+F)$ is also lc. Moreover, $B+F\leq B+B\,'$ is also a boundary. Hence Theorem \ref{big-nef-case} applies to $(X/Z\ni o,B+F)$ to generate an $n$-complement of $(X/Z\ni o,B+F)$ for some $n\in\NN$, which is also an $n$-complement of $(X/Z\ni o,B)$ by Lemma \ref{monotonicity}.
\end{proof}

\section{Proof of the Main Theorem \ref{main}}

Using the techniques and results introduced in the previous chapters, we are now ready to give the proof of Theorem \ref{main}.

%\begin{proof}
%Note $X$ is $\mathbb{Q}$-factorial from our assumption. Let $D=M+F$ be the Zariski decomposition and $M$ be the mobile part. Note we always have $M\leq D$. On one hand, $K_X+B$ is dlt and $\textup{Supp}D\subset\textup{Supp}(B-\lfloor B\rfloor)$, for $0<\delta\ll1$, $K_X+B+\delta D$ is also dlt. 
%It follows that $\delta M\sim_{\mathbb{R}}K_X+B+\delta M\leq K_X+B+\delta D$ is also dlt, hence lc. On the other hand, since $M$ is the mobile part of $D$, $M$ is nef. Now $\delta M\sim_{\mathbb{R}}K_X+B+\delta M$ is lc and nef, hence by the Abundance Theorem in dimension 2, $M$ is semi-ample.
%\end{proof}

%\begin{lem}\label{equality}
%Suppose $D_1\leq D_2$ are two semi-ample effective divisors on a normal surface $X$, and $\iota(X,D_1)=\iota(X,D_2)=1$. If $\varphi_1:X\rightarrow C_1$ and $\varphi_2:X\rightarrow C_2$ are the contractions induced by $D_1$ and $D_2$ respectively, then $\varphi_1=\varphi_2$.
%\end{lem}

%\begin{proof}
%It suffices to show that for an irreducible curve $C$ on $X$, $C.D_1=0\iff C.D_2=0$. Since components of $D_1$ are also components of $D_2$, they are contracted by $\varphi_2$, $D_1.D_2=0$. Since $D_1=\varphi_1^{*}L_1$ for some ample divisor $L_1$ on $C_1$,$D_1$ is numerically 
%equivalent to a positive multiple of the fiber $f_1$ of $\varphi_1$. Hence $D_2.f_1=0$. If $C.D_1=0$, $C$ is contained in some fiber of $\varphi_1$, there exists an effective divisor $E$ such that $C+E\equiv f_1$. Now that $D_2.(C+E)=0$, $D_2$ is nef and $C,E$ both effective, $D_2.C=0$ 
%follows. Similarly we can show that $C.D_2=0\Rightarrow C.D_1=0$.
%\end{proof}

\begin{proposition}\label{local-reduction-big-case}
\cite[Theorem 7.2.11]{Pr01}. Let $f:(X,B)\to Z\ni o$ be a contraction from a normal surface $X$ to a smooth curve $Z$ such that:
\begin{enumerate}
\item $K_X+B\sim_{\Rr}0/Z\ni o$ and $B$ is a boundary;
\item the general fiber of $f$ is a nonsingular rational curve.
\end{enumerate}
Then for any set of data $I,v,e,\epsilon$ as in Definition \ref{comp-restriction}, $(X/Z\ni o, B)$ has an $n$-complement for some $n\in\NN=\NN(I,v,e,\epsilon)$. Here $\NN$ is the same as in Corollary \ref{big-case}.
\end{proposition}

\begin{proof}
After a log terminal modification (Proposition \ref{dlt-modification}), we can assume that $X$ is smooth, and by adding $pf^*o$ where $p$ is the log canonical threshold we can assume $K_X+B$ is maximally lc, and it follows that the reduced part $C:=\lfloor B\rfloor$ is non-empty. Then choose a sufficiently general point $P$ such that $P$ is contained in only one component $C_1$ of $C$ and Supp$B$ is nonsingular at $P$. Then we get a higher model $f:X\to Z\ni o$ such that the fiber $F=f^{-1}o$ contains a (-1)-curve $E$ with multiplicity 1, $E$ is not contained in Supp$B$ and $(F-E).E=1$. It suffices to prove the result for this new model, since Lemma \ref{pushforwards} promises that $n$-complements are compatible with crepant pullbacks. By 2) in our assumption we have $K_X.F=-2$. Hence $K_X.(F-E)=-1$ and$(F-E)^2=(F-E).(-E)=-1$, we can contract $\Supp(F-E)$ and get $$f:X\stackrel{g}{\longrightarrow}Y\to Z.$$ Now that $g$ is birational, by Corollary \ref{big-case} $K_X+B$ has an $n$-complement $K_X+B^+$ over $o\,'=g(F-E)$. Numerically this is just $n(K_X+B^+)\equiv0$ over $F-E$. To show that this $n$-complement $K_X+B^+$ can be extended to an $n$-complement, it suffices to show $n(K_X+B^+)\sim0/Z\ni o$. But in our case linear equivalence is the same as numerical equivalence, it suffices to show $n(K_X+B^+)\equiv0$ over the fiber $F$ over $o$. Hence it remains to show $(K_X+B^+).E=0$. On one hand, $K_X.E=-1$ since $E$ is a (-1)-curve. On the other hand, $B^+.E=C_1.E=1$. The first equality holds since $E$ only intersects the component $C_1$ by our construction, and $C_1$ also has coefficient 1 in $B^+$ as in $B$ because $K_X+B^+$ is an $n$-complement of $K_X+B$. The second equality is again by our construction of $E$.
\end{proof}

\begin{proof}[Proof of Theorem 1.4]
We proceed in the following steps.

\textbf{Step 1}: In this step, we'll construct the finite set $\NN=\NN_1\cup\NN_2\cup\NN_3$ first. 

In our definition, there are three sets $\NN_i$, $i=1,2,3$ involved. They are defined as follows, which correspond to the three global cases in Step 4 below. Here by dimension we mean the Iitaka dimension of the divisor $M$, which will also be defined in Step 4.
\begin{enumerate}
\item 2-dimensional case: $\NN_1=\NN(I,v,e,\epsilon)$ as in Corollary \ref{big-case}.
\item 1-dimensional case: $\NN_2=\NN(lcm(I,J),v,e,\epsilon)$ as in Corollary \ref{boundedness-curve}. Here $J=J(\NN_1)$ is the adjunction index determined by the hyperstandard set $\Gamma(\NN_1,\Phi)$ as in Corollary \ref{adjunction-index-existence}.
\item 0-dimensional case: $\NN_3=\{N\}$, where $N\in\NN(lcm(I,I\,',12),v,e,\epsilon)$ is an integer under complementary restrictions with respect to given data (as in Definition \ref{comp-restriction}). Here $I\,'$ is the index determined by the hyperstandard set $\Gamma(\NN_1,\Phi)$ as in Theorem \ref{index-theorem-surfaces}.
\end{enumerate}
Note all $\NN_i$, $i=1,2,3$ consist of integers under complementary restrictions with respect to $I,v,e,\epsilon$ from our construction, it remains to check (2) from Theorem \ref{main}, namely we need to find an $n$-complement of $(X/Z\ni o,B)$ for some $n\in\NN$.

\textbf{Step 2}: In this step, we assume $X\to Z$ is a surjective contraction without loss of generality. 

We take the Stein factorization $X\to Z\,'\stackrel{\phi}{\longrightarrow}Z$ of $X\to Z$. Since $X_o$ is connected and $X\to Z\,'$ is a contraction, $o\,'=\phi^{-1}o$ is a single point. We claim that an $n$-complement $K_X+B^+$ of $(X/Z\,'\ni o\,', B)$ is also an $n$-complement of $(X/Z\ni o, B)$ (and the same holds for $\Rr$-complements). For this, we just need to check condition 3) in Definition \ref{n-comp}, namely $K_X+B^+\sim_n0/Z\ni o$ is equivalent to $K_X+B^+\sim_n0/Z\,'\ni o\,'$. We rephrase this linearly trivial condition: $K_X+B^+\sim_n0/Z\ni o$ means $n(K_X+B^+)\sim0$ near $X_o$, and the claim follows since the central fibers $X_{o\,'}$ and $X_o$ only differ in multiplicities. Based on this claim, we can then replace the pair $(X/Z\ni o, B)$ by $(X/Z\,'\ni o\,', B)$.

\textbf{Step 3}: In this step, we assume $(X,B)$ is dlt and $K_X+B\sim_{\Rr}0/Z\ni o$. It follows that $X$ is $\Qq$-factorial.

Indeed, by our assumption $(X/Z\ni o,B)$ has an $\Rr$-complement $(X/Z\ni o,B^+)$ with $B^+\geq B$. Since $(X,B^+)$ is lc and $B^+$ is a boundary, by Proposition \ref{dlt-modification}, we can take the minimal dlt modification $f:Y\to X$ of $K_X+B^+$. Let $B_Y^+$ be the strict transform of $B^+$. Then we have $K_Y+B_Y^++E=f^*(K_X+B^+)$, $E$ is a reduced exceptional divisor. By Lemma \ref{pushforwards}, an $n$-complement for $(Y/Z\ni o,B^+_Y+E)$ is also an $n$-complement for $(X/Z\ni o,B^+)$. By Lemma \ref{monotonicity}, an $n$-complement for $(X/Z\ni o,B^+)$ is also an $n$-complement for $(X/Z\ni o,B)$. Hence to find an $n$-complement for $(X/Z\ni o,B)$, we can replace $X$ by $Y$, $B$ by $B^+_Y+E$. The pair $(Y/Z\ni o,B^+_Y+E)$ is indeed a dlt 0-pair with a boundary $B^+_Y+E$, and is itself an $\Rr$-complement.

\textbf{Step 4 (global case)}: In this step, we show the existence of an $n$-complement of $(X,B)$ for some $n\in\NN$ in the global case $Z=pt$.

Take the low approximation $B_{\NN_1}$ of $B$ as in Construction \ref{low-approximation} and set $D=B-B_{\NN_1}$. By Construction \ref{low-approximation}, $D$ is effective and $\Supp D\subset\Supp (B-\lf B\rf)$. Hence for $0<\delta\ll1$, $(X,B+\delta D)$ is also dlt, hence lc. Applying Lemma \ref{semi-ampleness} to the pair $(X, B+\delta D)$, the mobile part $\delta M$ of $K_X+B+\delta D\sim_{\Rr}\delta D$ is effective and semiample. Hence a multiple of $M$ induces a morphism $\varphi: X\to T$. Then we distinguish three cases as follows based on $\iota(X, M)=\dim T$.

\textit{Case 1}: if $\dim T=2$, $(X,B)$ has an $n$-complement for some $n\in\NN_1$. 

In this case $M$ is big, hence $-(K_X+B_{\NN_1})\sim_{\Rr}D\geq M$ is also big, and the pair $(X,B_{\NN_1})$ has an $\Rr$-complement $(X,B)$. Hence by Corollary \ref{big-case}, $(X,B_{\NN_1})$ has an $n$-complement for some $n\in\NN_1$, which is also an $n$-complement for the pair $(X,B)$ by Lemma \ref{approximation-complements}.

\textit{Case 2}: if $\dim T=1$, $(X,B)$ has an $n$-complement for some $n\in\NN_2$.

Let $\varphi:(X,B)\to T$ be the contraction induced by a multiple of the effective divisor $M$. In this case $T$ is a smooth projective curve. Then $\varphi$ is a $0$-contraction of relative dimension 1: $(X,B)$ is generically lc over $T$, $B$ is generically a boundary over $T$, and $K_X+B\sim_{\Rr}0$ over $T$.  We claim that $\varphi$ contracts all components of $D=M+F$. Since $\varphi$ is induced by a multiple of the effective divisor $M$, $M=\varphi^*L$ for some ample divisor $L$ on $T$ and all the components of $M$ are contracted. For a component $C$ of the fixed part $F$, by Definition \ref{Zar-definition} we have $0=M.C=\varphi^*L.C=L.(\varphi)_*C=0$, $C$ is also contracted by $\varphi$. Then the horizontal part $B^h$ of $B=B_{\NN_1}+D$ over $T$ is contained in the support of $B_{\NN_1}$, whose coefficients belong to the hyperstandard set $\Gamma(\NN_1,\Phi)$. We can then apply Corollary \ref{adjunction-index-existence} and conclude that $\varphi$ has an adjunction index $J$. 

Now we can apply the canonical bundle formula to $\varphi$ and get $$(K_X+B)\sim_J\varphi^*(K_T+B_{\dive}+B_{\modu}).$$ Here $B_{\modu}$ is defined up to $J$-linear equivalence and $JB_{\modu}$ is integral. Since $B$ is a boundary, $B_{\dive}$ is also a boundary. By Theorem \ref{nef-modular-part}, $B_{\modu}$ is nef. Hence $K_T+B_T\sim_{\Rr}0$ and $\deg{B_T}\geq0$, where $B_T=B_{\dive}+B_{\modu}$. This implies $T$ is either elliptic or rational. If $T$ is elliptic, on one hand we have $B_T\sim_{\Rr}0$ and $B_{\dive}=0$ follows. But on the other hand $B=(B-M)+\varphi^*L$ with $B-M=B_{\NN_1}+F$ effective and $L$ ample, hence $B_{\dive}=(B-M)_{\dive}+L$ is big, contradiction. So $T\simeq\Pp^1$. We can then choose a suitable representative $B_{\modu}\geq0$ to make $B_T$ a boundary on $T$: for example, choose $JB_{\modu}=\sum P_i$ with $P_i$ distinct closed points on $T$ which are not in the support of $B_{\dive}$. Then by Corollary \ref{boundedness-curve}, $(T,B_T)$ has an $n$-complement for some $n\in\NN_2$, and by our construction of $\NN_2$, the adjunction index $J$ of $\varphi:(X,B)\to T$ divides $n$. Hence by Proposition \ref{lifting-complements}, we can pullback the $n$-complement of $(T,B_T)$ to an $n$-complement of $(X,B)$.

\textit{Case 3}: if $\dim T=0$, $(X,B)$ has an $n$-complement for some $n\in\NN_3$.
 
In this case $\iota(X, M)=0$ implies $M=0$. By Lemma \ref{semi-ampleness}, we have a birational contraction $g:X\to X\,'$ contracting the fixed part $F$. Now the pair $(X\,',B\,'=g_*B)$ is also lc since $K_X+B=g^*(K_{X\,'}+B\,')$ is crepant and $(X,B)$ is lc, $K_{X\,'}+B\,'=g_*(K_X+B)\sim_{\Rr}0$ and $B\,'=g_*B=g_*(B_{\NN_1}+F)=g_*B_{\NN_1}$ has coefficients in $\Gamma(\NN_1,\Phi)$. We can then apply Theorem \ref{index-theorem-surfaces} and conclude that $I\,'(K_{X\,'}+B\,')\sim0$. Hence $I\,'(K_X+B)\sim0$ and $(X,B)$ is an $n$-complement of itself for any $n$ divisible by $I\,'$, and in particular for some $n\in\NN_3$.

\textbf{Step 5 (local case)}: In this step, we show the existence of an $n$-complement of $(X/Z\ni o,B)$ for some $n\in\NN$ in the local case $\dim Z\geq1$.

In the local case, we can always replace $B$ by $B+t\varphi^*o$ where $\varphi:X\to Z$ such that $K_X+B+t\varphi^*o$ is maximally lc, since $K_X+B+t\varphi^*o\sim_{\Rr}0/Z\ni o$ also holds and by Lemma \ref{monotonicity} an $n$-complement of $(X/Z\ni o,B+t\varphi^*o)$ is also an $n$-complement of $(X/Z\ni o,B)$. Hence we assume $(X/Z\ni o,B)$ is maximally lc. Then we also have three cases based on the dimension of $Z$.

\textit{Case 1}: If $\varphi$ is birational, then $(X/Z\ni o, B)$ has an $n$-complement for some $n\in\NN_1$. This follows immediately from Corollary \ref{big-case} since $K_X+B\sim_{\Rr}0/Z\ni o$. 

\textit{Case 2}: If $Z$ is a curve and $B^h\neq0$, then $(X/Z\ni o, B)$ has an $n$-complement for some $n\in\NN_1$. This is just Proposition \ref{local-reduction-big-case} since a general fiber of $\varphi$ is a smooth rational curve in this case.

\textit{Case 3}: If $Z$ is a curve and $B^h=0$, then $(X/Z\ni o, B)$ has a $12$-complement, and in particular an $n$-complement for some $n\in\NN_3$.

In this case $\varphi:X\to Z$ is an elliptic fibration. If $K_X.C<0$ for some irreducible curve $C$ in the fiber $\varphi^{-1}o$, $C$ can't be the full fiber. Then $K_X.C<0$ and $C^2<0$, $C$ is a $(-1)$-curve and we can contract $C$. After contraction of all such $(-1)$-curves, we get $g:X\to X\,'$ where $K_{X\,'}$ is nef over $Z$. Since $K_{X\,'}.F=0$ for a general fiber $F$ of $X\,'/Z$, we have $K_{X\,'}\equiv0/Z$. Then $X\,'/Z$ is a minimal elliptic fibration with $K_{X\,'}+B\,'\sim_{\Rr}0/Z\ni o$ maximally lc, where $B\,'=g_*B$. By Proposition \ref{elliptic-fibration-complements}, $12(K_{X\,'}+B\,')\sim0/Z\ni o$ and since $K_X+B=g^*(K_{X\,'}+B\,')$ is crepant, $12(K_X+B)\sim0/Z\ni o$ also holds. It follows that $(X/Z\ni o, B)$ is an $n$-complement of itself for any $n$ divisible by 12, and in particular for any $n\in\NN_3$.
\end{proof}

\clearpage

\bibliographystyle{abbrv}

\end{document}